\documentclass[A4paper, 12pt]{article}
\usepackage{amssymb,amsthm,amsmath,amsfonts, amscd}
\usepackage{verbatim}

\numberwithin{equation}{section}

\def \P {\mathbb P}
\newtheorem{theorem}{Theorem}[section]
\newtheorem{lem}[theorem]{Lemma}

\theoremstyle{definition}

\DeclareMathOperator{\meas}{\meas}
\usepackage[top=3cm, bottom=2cm, left=3cm, right=1.5 cm]{geometry}

\makeatletter
\renewcommand\@biblabel[1]{#1.}
\makeatother

\makeatletter
\long\def\@makecaption#1#2{%
  \vskip\abovecaptionskip
  \sbox\@tempboxa{#1.~#2}%
  \ifdim \wd\@tempboxa >\hsize
    #1.~#2\par
  \else
    \global \@minipagefalse
    \hb@xt@\hsize{\hfil\box\@tempboxa\hfil}%
  \fi
  \vskip\belowcaptionskip}
\makeatother
\reversemarginpar

\begin{document}

\bigskip

\noindent {\large \bf Asymptotic Efficiency of Goodness-of-fit Tests for the Power  }

\medskip

\noindent {\large \bf Function  Distribution Based on Puri--Rubin Characterization.}

\bigskip
\bigskip

\centerline{ \bf Ya. Yu. Nikitin, K. Yu. Volkova\footnote{ e-mail \, yanikit47@gmail.com, \quad efrksenia@yandex.ru      }}

\bigskip

\centerline{ Saint-Petersburg State University}

\bigskip
\bigskip
\bigskip

\parbox{15cm}{ \small  We construct integral and supremum type goodness-of-fit tests
for the family of power distribution functions. Test statistics are functionals
of $U-$empirical processes and are based on the classical characterization of
power function distribution family belonging to Puri and Rubin. We describe the
logarithmic large deviation asymptotics of test statistics under
null-hypothesis, and calculate their local Bahadur efficiency under common
parametric alternatives. Conditions of local optimality of new
statistics are given.}

\bigskip


\bigskip


 \section {Introduction.}

Testing goodness-of-fit for parametric families of distributions remains one of important and interesting statistical problems.
Let ${\cal P}$ be the family of {\it power function distributions} with the distribution functions (d.f.)
\begin{equation}
\label{DF}
F(x) = x^{\lambda}, \ x \in (0,1), \  \lambda>0.
\end{equation}
 It is the member of the beta family and is the "inverse" of Pareto distribution. Power function distribution often appears  in applications, e.g.  in the study of service periods  of queueing systems \cite{andrade},  in the economic models of lead-time and pricing \cite{kazaz}, and in the reliability of electric systems \cite{meni}.

We are interested in goodness-of-fit tests for this family which are independent of unknown parameter $\lambda.$ As far as we know,
the only attempt to build such tests has been traced by Martynov \cite{Mart}, \cite{Mart2} who proposed to use the well-known  Durbin's approach \cite{durbin} based on the empirical process with estimated parameters.

In this paper we develop completely different way introducing and analy\-zing two tests based on characterization of the power function distribution. Consider the following characte\-ri\-za\-tion by Puri and Rubin \cite{PR}:\\
\indent \emph{ Let $X$ and $Y$ be i.i.d. non-negative random variables. Then the equality in law of $X$ and $\min(\frac{X}{Y}, \frac{Y}{X})$ takes place iff  $X$ has some d.f. from the family $\cal P$.}

\smallskip

It should be noted that this result has been obtained by means of monotonic transfor\-ma\-tion from the characterization of exponentiality obtained in  \cite{PR}. This is a common and traditional method to restate the characterization theorems. However, as noted in \cite[p.169]{gala}, "while a property can be interesting for one distribution, it may lose its appeal after a transformation." \, But we find the characterization of the power function distribution family stated above rather convenient for goodness-of-fit purposes.

Let $X_1,X_2,\dots $ be i.i.d. observations with the continuous d.f. $F.$ We are interested in testing the hypothesis $H_0: F \in {\cal P}$ against the general alternative $H_1: F  \notin {\cal P},$ assuming, however, that the alternative d.f. is also concentrated on $(0,1).$

Let  $ F_n(t)=n^{-1}\sum_{i=1}^n\textbf{1}\{X_i<t\}, t \in R^1,$  be the usual empirical d.f. based on the sample $X_1,\dots,X_n.$
According to the Puri-Rubin characterization we introduce the so-called  $U$-empirical d.f., see \cite{Jan}, \cite{Kor}, by
\begin{gather*}
H_n (t)={n \choose 2}^{-1}\sum_{1\leq i<j \leq n}\textbf{1}\{\min(\frac{X_i}{X_j}, \frac{X_j}{X_i})< t\}, \quad t\in (0,1).
\end{gather*}

Consider two statistics which can be used for testing $H_0$ against $H_1:$
\begin{align}
I_n^{PR}&=\int_{0}^{1} \left(H_n(t)-F_n(t)\right)dF_n(t),\label{I_n_PR}\\
D_n^{PR}&=\sup_{t \in [0,1]}\mid H_n(t)-F_n(t)\mid\label{D_n_PR}.
\end{align}

The first of this statistics is of integral type and resembles the classical $\omega_n^1$-statistic while the second is of Kolmogorov type.
We will  describe their limit distributions under $H_0$ and we will calculate their local Bahadur efficiency under certain parametric alternatives. To this end we need their rough large deviation asymptotics under $H_0.$  Moreover, we will  discuss the conditions of their local optimality in the Bahadur sense.

For basic information on Bahadur theory we refer to \cite{Bahadur},  \cite{anirban} and \cite{Nik}. This type of efficiency is most pertinent in our problem as the Kolmogorov type statistics have non-normal distribution and hence the Pitman approach is not applicable.

In Bahadur theory the measure of efficiency of the sequence of statistics $\{T_n\}$ is the exact slope $c_T(\theta)$ describing the exponential decrease rate of ther $P-$values under the alternative. It is well-known (it is the so-called Bahadur-Raghavachari inequality \cite{Bahadur}, \cite{Nik})  that always
$$ c_T (\theta) \leq 2 K(\theta), $$ where $K(\theta)$ is the Kullback-Leibler "distance" between the null-hypothesis and the alternative which is indexed by real parameter $\theta.$
Therefore we may define the local Bahadur efficiency as
$$
eff(T): = \lim_{\theta \to 0}  c_T(\theta) / 2K(\theta).
$$

\section{Statistic $I_{n}^{PR}$}

The statistic $I_{n}^{PR}$ is asymptotically equivalent to the $U$-statistic of degree 3 with the centred kernel
$$
\Psi_{PR}(X,Y, Z)=\frac13\left (\textbf{1}\{\min(\frac{X}{Y}, \frac{Y}{X})< Z\}+\textbf{1}\{\min(\frac{X}{Z}, \frac{Z}{X})< Y\}+\textbf{1}\{\min(\frac{Y}{Z}, \frac{Z}{Y})< X\}\right)-
\frac12.
$$

Note that both statistics  $I_{n}^{PR}$ and $D_n^{PR}$ under $H_0$ are invariant with respect to the change of variable $ X \to X^{1/\lambda}.$ Therefore we may take $\lambda =1,$ i.e.  we can assume that the initial sample is uniform on $(0,1).$

It is well-known, see, e.g. \cite{Hoeffding}, \cite{Kor} that non-degenerate $U$- and $V$-statistics are asymptotically normal.  To prove that
 the kernel $\Psi_{PR}(X,Y,Z)$ is non-degenerate, let calculate its projection
 $\psi_{PR}.$  For fixed $X=s$ we have
$$
\psi_{PR}(s) := E(\Psi_{PR}(X, Y,Z)\mid X=s) =\frac23\P\{\min(\frac{s}{Y}, \frac{Y}{s})< Z\}+\frac13\P\{\min(\frac{Y}{Z}, \frac{Z}{Y})< s\}-\frac12.
$$
First probability can be evaluated as follows:
$$
\P\{\min(\frac{s}{Y}, \frac{Y}{s})<
Z\}=1-\int_{s}^{1}\frac{s}{y} \ dy-\int_{0}^{s}\frac{y}{s} \ dy=1+s\ln{s}-\frac12 s,
$$
and it results from the above characterization that
$$
\P\{\min(\frac{Y}{Z}, \frac{Z}{Y})< s\}=\P\{Y<s\}=s, \ 0 \leq s \leq 1.
$$
Hence we get the final expression for the projection of the kernel:
\begin{equation}
\label{projection}
\psi_{PR}(s) = \frac16+\frac23  s \ln{s}, \ 0 \leq s \leq 1.
\end{equation}

The variance of the projection is  given by
$$\Delta^2_{PR} = \int_{0}^{1} \psi_{PR}^2 (s) ds =\frac{5}{972},
$$
and is positive. Hence, the kernel $\Psi_{PR}(X,Y, Z)$
is non-degenerate. Due to Hoeffding's  theorem \cite{Hoeffding}, \cite{Kor}
$$\sqrt{n}I_{n}^{PR} \stackrel{d}{\longrightarrow}{\cal
{N}}(0,\frac{5}{108}).$$

The kernel  $\Psi_{PR}$ is centred, non-degenerate and bounded. Applying the theorem on large deviations for non-degenerate
 $U$-statistics from \cite{nikiponi}, see also \cite{anirban},  \cite{Nikolm},  we get:
\begin{theorem}
 For $a>0$ it holds true that
$$
\lim_{n\to \infty} n^{-1} \ln \P ( I_n^{PR} >a) = - f (a),
$$
where the function $f$ is analytic for sufficiently small $a>0,$ and that $$ f
(a) \sim \frac{a^2}{18 \Delta^2_{PR}} = \frac{54}{5} \, a^2, \quad \mbox{при} \,
\, a \to 0.
$$
\end{theorem}

In case of uniform null-distribution, and more generally, for the power function distribution,
there are no accepted standard alternatives. Therefore we consider in this paper
three alternatives: the contamination alternative and two other unnamed alternatives concentrated on $(0,1).$
The expressions of these alternative d.f.'s are as follows:
$$
\begin{array}{ll}

\medskip

G_1(x,\theta)=(1-\theta)x +\theta x^r, \,  0\leq \theta \leq 1,\ r>1, \
x \in (0,1).\\

\medskip

G_2(x,\theta)= x  - \theta \sin(\pi x) , \,
0 \leq \theta< 1/\pi, \ x \in (0,1).\\

\medskip

G_3(x,\theta)= x  + \theta \int_0^x \left( \frac{1}{6} + \frac{2}{3}  y \ln y \right) dy,
 \, 0 \leq \theta \leq 1, \ x \in (0,1).

 \medskip

\end{array}
$$
\noindent The formulas for corresponding densities $g_j(x,\theta), j=1,2,3$ are straightforward.

We will need in the sequel the expressions as $\theta \to 0$ of the Kullback-Leibler "distance" between the null-hypothesis and the considered alternatives. Note that the null-hypothesis is the composite one. We will establish now some general form for this distance as $\theta \to 0.$
\begin{lem}
\label{K_eva}
Let $g(x,\theta)$ be any alternative density on $(0,1)$ which is sufficiently regular so that any differentiation under the sign of integral  in the proof is justifiable and the Kullback-Leibler information (\ref{Kull}) is well-defined. Put
\begin{equation}
\label{Kull}
K(\theta) = \inf_{\lambda>0} \int_0^1 \ln \frac{g(x,\theta)}{\lambda x^{\lambda - 1}} g(x,\theta) dx.
\end{equation}
Then
\begin{equation}
\label{Kull_loc}
2K(\theta)\sim  \theta^2 \left[\int_{0}^{1}(g^{'}_{\theta}(x, 0))^2dx
-\left( \int_{0}^{1} g^{'}_{\theta}(x, 0) \ln{x} dx\right)^2  \right], \, \theta \to 0.
\end{equation}
\end{lem}
\begin{proof}

The infimum in (\ref{Kull}) is attained for  $\lambda = - (\int_{0}^{1} g(x, \theta)\ln{x} dx)^{-1}$ and equals
\begin{gather}
K(\theta)=\int_{0}^{1}g(x, \theta)\ln{g(x, \theta)} dx+ \int_{0}^{1}g(x, \theta)\ln{x} dx+ \ln(-\int_{0}^{1}g(x, \theta)\ln{x} dx)+1.\label{K_PR}
\end{gather}
As $\theta \to 0$ the function $ K(\theta)$ has the following form:
$$
K(\theta) ~ \sim K(0)+ K'(0) \cdot \theta + \frac12 K''(0) \cdot \theta^2.
$$
It is easy to see that  $K(0)=0$ and that $K'(0)=0.$

Differentiating in $\theta$  two times the right-hand side of  \eqref{K_PR} we get
\begin{gather*}
K''(\theta)=\int_{0}^{1}g''_{\theta^2}(x, \theta)(1+\ln{g(x, \theta)}+\ln{x})dx+ \int_{0}^{1}\frac{{g'_{\theta}}^2(x, \theta)}{g(x, \theta)}dx + \\
+ \int_{0}^{1}g''_{\theta^2}(x, \theta) \ln{x} \ dx \left(\int_{0}^{1}g(x, \theta)\ln{x} dx\right)^{-1}
-\left( \frac{\int_0^1 g'_{\theta}(x, \theta)\ln{x} dx}{\int_{0}^{1}g(x, \theta)\ln{x} dx}\right)^2.
\end{gather*}
Substituting  $\theta =0,$ one obtains the required expression.
\end{proof}

Let calculate  the local Bahadur exact slope
and the local efficiency of the sequence of  statistics  $I_n^{PR}$  for the  alternative  d.f. $G(x,\theta)$ and the density $ g(x,\theta)$ assuming their regularity and the possibility of differentiating under the integral sign.  These conditions are valid for all three alternatives we consider.
Denote also $h(x) = g'_{\theta}(x,0). $  Note that $ \int_0^1 h(x) dx =0.$

 According to the Law of Large Numbers for $U$-statistics
\cite{Kor} the limit in probability of the sequence $I_n^{PR} $ under any such alternative
is equal as $\theta \to 0$ to
\begin{multline*}
b_I(\theta)=\P_{\theta} (\min(\frac{X}{Y}, \frac{Y}{X})< Z) - \frac12
 = 2\int_{0}^{1}g(z, \theta)dz  \int_{0}^{1}g(y, \theta)G(yz,
\theta)dy-\frac12 \sim J(0)+ J'(0) \cdot \theta. \label{bI1_PR}
\end{multline*}

 It is easy to see that $J(0)=0,$  while $ J'(0) = 2\int_{0}^{1}h(z) z dz+2\int_{0}^{1} dz\int_{0}^{1} dy\int_{0}^{yz}h(x)dx. $
Chan\-ging two times the order of integration, we get
$$
\int_{0}^{1} dz\int_{0}^{1} dy \int_{0}^{yz}h(x)dx=\int_{0}^{1} dz\int_{0}^{z}h(x)\left(1-\frac{x}{z}\right)dx= \int_{0}^{1}h(x)\left( x \ln{x} - x \right)dx .
$$
It follows therefore that
\begin{equation}
\label{PR}
b_I(\theta; PR) \sim
3\theta  \int_{0}^{1} \psi_{PR}(x)h(x)ds.
\end{equation}

\textbf{Contamination alternative.} After elementary calculations we get by (\ref{PR}) as $\theta \to 0$ that
$ b_I(\theta) \sim (r-1)^2/2(r+1)^2 \cdot \theta. $ Therefore the local exact slope of the sequence of statistics
$I_n^{PR}$ as $\theta \to 0$ admits the representation $$ c_I(\theta;
PR) \sim \frac{108}{5} \, b_I^2(\theta; PR) = \frac{27(r-1)^4}{5(r+1)^4}\
\theta^2.$$

It is easy to show using (\ref{Kull_loc}) that for the alternative d.f. $G_1$
\begin{equation}
\label{Kull_cont}
 2K(\theta)\sim \frac{(r-1)^4}{r^2 (2r-1)}\theta^2, \, \theta \to 0.
\end{equation}

Hence the local Bahadur efficiency of our test is equal to
$$eff(r;I^{PR})=\lim_{\theta \to 0}\frac{c_I(\theta; PR)}{2K(\theta)}=\frac{27(2r-1)r^2}{5(r+1)^4}\ .$$

This efficiency is reasonably high for moderate values of $r$, its maximum is attained for $r=2+\sqrt{3}$ and equals $0.970.$

\textbf{Second alternative}
 The calculation of local Bahadur efficiency
in the case of alternative $G_2$ is quite similar.  We have by (\ref{PR}) $ b_I(\theta; PR) \sim 0.224 \cdot  \theta^2$, so that the local exact slope of $I_n^{PR}$ as $\theta \to 0$ admits the representation $ c_I(\theta; PR) \sim 1.083 \cdot \theta^2.$

According to (\ref{Kull_loc}), the Kullback-Leibler information in this case satisfies
\begin{equation}
\label{K_2}
2 K(\theta)\sim 1.505 \cdot \theta^2, \, \theta \to 0.
\end{equation}
Consequently, the local Bahadur efficiency of our test is $eff(I^{PR})= 0.719.$

\textbf{ Third alternative.} In the case of the third alternative the calculations are alike, and we obtain
after some calculations that $ c_I(\theta;PR) \sim \frac{5\theta^2}{972}$ as $\theta \to 0.$ The Kullback-Leibler information also satisfies in this case the  relation
\begin{equation}
\label{K_3}
 2K(\theta)\sim \frac{5\theta^2}{972}, \, \theta \to 0.
 \end{equation}

Therefore, the local Bahadur efficiency is equal to  $1,$ and the integral test is locally optimal in Bahadur sense \cite[Ch. 6]{Nik}.
We will return to the cause of this phenomenon in the last section.

\begin{table}[!hhh]\centering

\medskip

\caption{\it Local Bahadur efficiency for the statistic $I_n^{PR}.$ }

\medskip

\begin{tabular}{|c|c|}
\hline
Alternative & Efficiency\\
\hline
$G_1$ & 0.970, $r \approx$ 3.7\\
$G_2$ & 0.719\\
$G_3$ & 1. 000\\
\hline
\end{tabular}
\end{table}

\section{Statistic $D_n^{PR}$}

Now we consider the Kolmogorov type statistic (\ref{D_n_PR}). In this case
for fixed  $t$ the difference $H_n (t) -
F_n (t)$  is a family of $U$-statistics with the kernels
$$
\Xi_{PR}(X, Y;t)= \textbf{1}\{\min(\frac{X}{Y}, \frac{Y}{X})<
t\}-\frac12 \textbf{1}\{X <t\}-\frac12 \textbf{1}\{Y <t\} ,
$$
depending on  $t\in (0,1).$ The projection of this kernel for fixed
$t\in (0,1)$ has the form
$$
\xi_{PR}(s;t) := E\left( \Xi_{PR}(X, Y;t)|X=s\right)=
\P\{\min(\frac{s}{Y}, \frac{Y}{s})< t\}-\frac12 \textbf{1}\{s
<t\}-\frac12 \P\{Y <t\}.
$$
After easy calculations we get
\begin{equation}
\label{proj_d}
\xi_{PR}(s;t) =  \textbf{1}\{s <t\} (\frac12-\frac{s}{t})+ st -\frac12 t.
\end{equation}

Now let calculate the variance $ \delta^2_{PR}(t)$ of this projection.  We have after
some simple calculations
\begin{equation}
\label{proj}
\delta^2_{PR}(t):= E \xi_{PR}^2 (X_1;t) = \frac{1}{12}t(1+t-2t^{2}), \, 0 < t < 1.
\end{equation}

It is easy to see that the supremum of the function $\delta^2_{PR}(t)$
is attained in the point
$t^{*}=\frac{1+\sqrt{7}}{6}$  and equals  $\delta^2_{PR}\approx 0.044.$
Hence our family of kernels  $\Xi_{PR}(X;t)$ by \cite{Nikolm}  is non-degenerate.

The limiting distribution of the statistic  $D_n^{PR}$ is unknown. Using the mehods
developed in \cite{Silv}, one can show that the $U$-empirical process
$$\eta_n(t) =\sqrt{n} \left(H_n (t) - F_n (t)\right), \ t\in
(0,1),
$$ converges weakly as $n \to \infty$ to some centered Gaussian process $\eta(t)$ with
complicated covariance. Then the sequence of statistics $\sqrt{n} D_n^{PR}$ converges in distribution to the random variable $\sup_t |\eta(t)|$
whose distribution we are not able to find. Hence we suggest to use statistical modelling to evaluate  the critical values for the statistics  $D_n^{PR}.$

The family of kernels $\{\Xi_{PR}(X, Y;t)\}, t\in (0,1)$ is not only centred but bounded. Using the results of \cite{Nikolm} on
large deviations of  families of non-degenerate $U$-statistics, we obtain the following result.
\begin{theorem}
For sufficiently small $a>0$ it holds true that
$$
\lim_{n\to \infty} n^{-1} \ln \P ( D_n^{PR} >a) = - k(a),
$$
where the function $k$ is analytic, and moreover  $$
k(a) = \frac{ a^2}{8 \delta^2_{PR}}(1 + o(1)) \sim 2.84 \ a^2, \,
\mbox{as} \, \, a \to 0.
$$
\end{theorem}

\textbf{Contamination alternative.}  Let calculate the local Bahadur slope
and local efficiency of the statistic \eqref{D_n_PR}  for the alternative d.f.
$G_1(x,\theta).$  By Glivenko-Cantelli theorem for  $U$-empirical d.f.'s \cite{Jan} the limit of $D_n^{PR}$  almost surely
under any alternative   is equal as $\theta \to 0$ to
\begin{align}\label{bD1_PR}
  b_D(t,\theta; PR)=\sup_{0 \leq t \leq 1}|2\int_{0}^{1}g(y, \theta)G(ty, \theta)dy-G(t,\theta) |.
\end{align}
Assuming the regularity of the alternative d.f., we can deduce
\begin{equation}\label{asympt_b}
  b_D(t,\theta; PR) \sim 2\sup_{0 \leq t \leq 1}|\int_0^1 \xi_{PR}(s;t) h(s) ds |\cdot \theta,
\end{equation}
where $\xi_{PR}(s;t) $ is from (\ref{proj_d}). Applying this formula we get for our alternative
$$
b_D(t,\theta; PR) \sim \frac{(r-1)^2}{r+1}r^{-\frac{r}{r-1}}\theta, \, \theta \to 0.
$$
Hence, the local exact slope of the sequence of statistics
$D_n^{PR}$ as $\theta \to 0$ admits the representation $$ c_D(\theta;
PR) \sim \frac{5.68(r-1)^4r^{-\frac{2r}{r-1}}}{(r+1)^2}\ \theta^2.$$
The Kullback-Leibler information satisfies  (\ref{Kull_cont}).
Hence the local Bahadur efficiency of our test is equal to
$$eff(r;D^{PR})= \frac{5.68 (2r-1)r^{-\frac{2}{r-1}}}{(r+1)^2} .$$

It can be shown that the maximal value of the local efficiency for the sequence $\{D_n^{PR}\}$  is attained for $r=4.64$ and
is  equal to $0.636$ while its values for $3 \leq r \leq 12$ are larger than 0.5.

\textbf{Second alternative.}  The calculation of local Bahadur efficiency
in the case of alternative $G_2$ is quite similar.  We have as $\theta \to 0$ by (\ref{asympt_b})
$ b_D(\theta; PR) \sim 0.367\cdot  \theta.$
Therefore the local exact slope of $D_n^{PR}$  admits the asymptotics $ c_D(\theta;PR) \sim 0.765 \cdot \theta^2.$
The Kullback-Leibler information in this case is  given by (\ref{K_2}). Hence the local Bahadur efficiency of our test is $eff(D^{PR})= 0.508.$

\textbf{ Third alternative.}  In this case we get for  the alternative d.f.
$G_3(x,\theta)$ as $\theta \to 0$ that  $ b_D(\theta; PR) \sim  0.0249 \cdot \theta.$
Hence the local exact slope of the sequence of statistics
$D_n^{PR}$ as $\theta \to 0$ admits the representation $ c_D(\theta;
PR) \sim 0.00352 \cdot \theta^2.$
We know that the Kullback-Leibler information in this case
satisfies (\ref{K_3}). Thus the local Baha\-dur efficiency of our test
is equal to $0. 685.$
\begin{table}[!hhh]\centering
\medskip
\caption{\it Local Bahadur efficiency for statistic $D_n^{PR}.$ }
\begin{tabular}{|c|c|}
\hline
Alternative & Efficiency \\
\hline
$G_1$ & 0.636, for $ r \approx$ 4.64  \\
$G_2$ & 0.508\\
$G_3$ & 0.685\\
\hline
\end{tabular}
\end{table}

It is seen that the Kolmogorov statistic
is less efficient than the integral statistic $I_n^{PR}$ as usually in goodness-of-fit testing \cite{Nik}.

\section{Conditions of local asymptotic optimality}
In this section we are interested in conditions of local asymptotic optimality (LAO) in Bahadur sense
for both sequences of statistics $I_n^{PR}$  and $D_n^{PR}.$ This means to describe the local structure of the
alternatives for which the given statistic has maximal potential local efficiency so that the relation
$$
c_T(\theta) \sim 2 K(\theta),\, \theta \to 0,
$$
holds, see \cite{Nik},\cite{gini}.  Such alternatives form the domain of LAO for the given sequence of statistics.

Consider the functions
\begin{gather*}
H(x)=G^{'}_{\theta}(x,\theta)\mid
_{\theta=0},\quad
h(x)=g^{'}_{\theta} (x,\theta)\mid _{\theta=0}.
\end{gather*}
We will assume that the following regularity conditions are true, see also \cite{gini}:
\begin{gather}
\int_{0}^{1} h^2(x)dx <  \infty  \quad  \mbox{where} \quad  h(x)=H'(x), \, \label{PR1} \\
 \frac{\partial}{\partial\theta}\int_{0}^{1} g(x,\theta) q(x)dx \mid
_{\theta=0} \ = \ \int_{0}^{1} h(x)q(x)dx \quad  \forall q \in L_1 (0, 1).  \label{PR2}
\end{gather}
Denote by $\cal G$ the class of densities  $ g(x,\theta)$  with d.f.'s $G(x,\theta),$ satisfying the regularity conditions (\ref{PR1}) - (\ref{PR2}). We are going to deduce the LAO conditions in terms of the function $h(x).$

For alternative densities from  $\cal G$  the arguments of Lemma \ref{K_eva}
are true, hence the asymptotics
$$
2K(\theta)\sim \left\{\int_{0}^{1} h^2(x)dx -\left( \int_{0}^{1}h(x)\ln{x} dx\right)^2\right\}\theta^2,\quad  \theta \to 0,
$$
is valid.

First consider the integral statistic  $I_n^{PR}$  with the kernel
$\Psi_{PR}(x, y, z)$ and its projection $\psi_{PR}(x) =  \frac16+\frac23 x \ln{x}.$
Let introduce the auxiliary function
$$
h_0(x) = h(x) - (\ln{x}+1)\int_0^\infty \ln{u} h(u) du.
$$

Simple calculations show that
\begin{gather*}
\medskip
\int_{0}^{1} h^2(x)dx -\left(\int_{0}^{1} h(x)\ln{x} dx\right)^2= \int_0^{1} h_0^2(x) dx,\\
\int_{0}^{1} \psi_{PR}(x)h(x)dx = \int_{0}^{1} \psi_{PR}(x)h_0(x)dx,\\
\int_{0}^{1} \xi_{PR}(x; t)h(x)dx = \int_{0}^{1} \xi_{PR}(x; t )h_0(x)dx \quad  \text{for any } \, t \in (0,1).
\end{gather*}

Hence the local asymptotic efficiency by (\ref{PR}) takes the form
\begin{gather*}eff(I_n^{PR})= \lim_{\theta \to 0} b_I^2(\theta; PR) / \left(9\Delta^2_{PR} \cdot 2K(\theta)\right) =\\
= \left(\int_{0}^{1} \psi_{PR}(x)h_0(x)dx\right)^2/\left(
\int_{0}^{1}\psi_{PR}^2(x) dx \cdot   \int_0^{1} h_0^2(x) dx
 \right).
\end{gather*}

By Cauchy-Schwarz inequality
we obtain that  the expression in the right-hand side is equal to 1 iff
 $h_0(x)=C_1\psi_{PR}(x)$  for some constant $C_1>0,$
so that $h(x) =C_1\psi_{PR}(x)+ C_2 (\ln{x}+1)$  for some constants  $C_1>0$ and $C_2.$ The set of distributions
for which the function $h(x)$ has such form generate the domain of LAO in the class $\cal G$.
The example of such alternative is the density $g(x,\theta)$ which for  small $\theta > 0$  satisfies the formula
\begin{equation}
\label{special0}
g(x,\theta)=1+\theta\left( \frac16+\frac23  x \ln{x}\right), \  0 \leq x \leq 1.
\end{equation}
This explains why the  third alternative leads to asymptotic optimality of the test based on $I_n^{PR}.$ It is in perfect agreement with the findings of the paper \cite{peau} where similar problems were solved for the simple null-hypothesis.

Now let consider the Kolmogorov type statistic  $D_n^{PR}$  with the family of kernels  $\Xi_{PR}(X, Y;t)$ and their
projections $\xi_{PR}(x;t)=\textbf{1}\{x <t\} (\frac12-\frac{x}{t})+xt-\frac{t}{2}.$
In this case it is easy to see that the following asymptotics is true:
\begin{gather}
 b_D(\theta; PR) \sim 2 \theta \sup_{t\in (0,1]}
\mid  \int_{0}^{1}  \xi_{PR}(x;t)h_0(x)dx\mid. \label{bPR_LAO}
\end{gather}

Hence the local efficiency takes the form
\begin{multline*}
eff(D^{PR})= \lim_{\theta \to 0} \left [ b_D^2(\theta;PR)/ \sup_{t\in(0,1)}\left(
4\delta^2_{PR}(t)\right)\cdot 2K(\theta) \right ] =\\ = \sup_{t\in(0,1)}\left(
\int_{0}^{1}  \xi_{PR}(x;t)h_0(x)dx\right)^2 / \ \sup_{t\in (0,1)} \left(
\int_{0}^{1} \xi_{PR}^2 (x,t) dx \cdot \int_{0}^{1}
h_0^2(x)dx\right) \leq 1.
\end{multline*}

We can apply once again  the Cauchy-Schwarz inequality to the integral in \eqref{bPR_LAO}. It follows that the
sequence of statistics $D_n^{PR}$ is locally asymptotically optimal, and $eff(D^{PR})=1$  iff  $h(x)=C_3\xi_{PR}(x, t_0)+C_4(\ln{x}+1)$  for
$t_0= \arg \sup_{t\in (0,1)}\delta_{PR}^2(t)=\frac{1+\sqrt{7}}{6}$ and some constants $C_3>0$ and  $C_4.$
The distributions with such  $h(x)$ form the domain of LAO in the class  $\cal G$. The simplest example of
such alternative density  $g(x,\theta)$  which for small $\theta > 0$ is given by the formula
\begin{equation}
\label{special1}
g(x,\theta)=1+\theta
\left(\textbf{1}\{x <t_0\} \left(\frac12-\frac{x}{t_0}\right)+xt_0-\frac{t_0}{2}\right), 0\leq x \leq 1,  \, \text{ where }t_0= \frac{1+\sqrt{7}}{6}.
\end{equation}

Hence we see that there exist  special alternative densities (\ref{special0}) and (\ref{special1}) of relatively simple form for which our sequences of statistics are locally asymptotically optimal. This stresses their merits and potential utility.

\medskip

{\small \textbf{Acknowledgement} The research of  both authors was supported by the grant of RFBR  10-01-00154,  by the grants  NSh-1216.2012.1 and FZP 2010-1.1-111-128-033. }

\end{document}